\documentclass[leqno]{amsart}
\usepackage{caption}
\usepackage{amsmath,amssymb,amsthm,mathtools,mathrsfs,tikz}
\usepackage{environ}
\usepackage{stmaryrd}
\usetikzlibrary{arrows}
\usetikzlibrary{positioning}
\usetikzlibrary{decorations.text}
\usetikzlibrary{decorations.pathmorphing}
\usetikzlibrary{decorations.pathreplacing}
\makeatletter
\newsavebox{\@brx}
\newcommand{\llangle}[1][]{\savebox{\@brx}{\(\m@th{#1\langle}\)}%
  \mathopen{\copy\@brx\kern-0.5\wd\@brx\usebox{\@brx}}}
\newcommand{\rrangle}[1][]{\savebox{\@brx}{\(\m@th{#1\rangle}\)}%
  \mathclose{\copy\@brx\kern-0.5\wd\@brx\usebox{\@brx}}}
\makeatother
\makeatletter
\newsavebox{\measure@tikzpicture}
\NewEnviron{scaletikzpicturetowidth}[1]{%
  \mathrm{Def}\tikz@width{#1}%
  \mathrm{Def}\tikzscale{1}\begin{lrbox}{\measure@tikzpicture}%
  \BODY
  \end{lrbox}%
  \pgfmathparse{#1/\wd\measure@tikzpicture}%
  \edef\tikzscale{\pgfmathresult}%
  \BODY
}
\makeatother
\usepackage{bbm}
\DeclarePairedDelimiter\norm{\lvert}{\rvert}
\DeclarePairedDelimiter\inner{\langle}{\rangle}

\usepackage{bm}
\usepackage[colorlinks=true,linkcolor=blue,citecolor=blue]{hyperref}
\usepackage{enumitem}
\usepackage{csquotes}

\numberwithin{equation}{section}
\newcounter{intro}

		\newtheorem{introthm}[intro]{Theorem}

		\newtheorem{thm}[equation]{Theorem}
		\newtheorem{lem}[equation]{Lemma}
		\newtheorem{cor}[equation]{Corollary}

\theoremstyle{remark}
		
\theoremstyle{definition}
		
		\newtheorem{exam}[equation]{Example}

\newcommand{\irr}{\mathrm{Irr}}

\title{A Frobenius group analog for Camina triples}

\author{Shawn T. Burkett}
\address{} \email{shawn.t.burkett@gmail.com}

\author{Mark L. Lewis}
\address{Department of Mathematical Sciences, Kent State University, Kent,
Ohio 44242, U.S.A.} \email{lewis@math.kent.edu}

\date{\today}

\begin{document}

\begin{abstract}
Frobenius groups are an object of fundamental importance in finite group theory. As such, several generalizations of these groups have been considered. Some examples include: A Frobenius--Wielandt group is a triple $(G,H,L)$ where $H/L$ is {\it almost} a Frobenius complement for $G$;  A Camina pair is a pair $(G,N)$ where $N$ is {\it almost} a Frobenius kernel for $G$; A Camina triple is a triple $(G,N,M)$ where $(G,N)$ and $(G,M)$ are {\it almost} Camina pairs. In this paper we study triples $(G,N,M)$ where $(G,N)$ and $(G,M)$ are {\it almost} Frobenius groups. 
\end{abstract}

\maketitle

\section{Introduction}

Frobenius groups are ubiquitous in finite group theory as they arise very naturally in a variety of situations.  In addition, Frobenius groups have been generalized in a number of different ways.  In this paper, we present a new generalization of Frobenius groups.

Recall that if $G$ is a Frobenius group, then there is a normal subgroup $N$ that has complement $H$ in $G$ so that $(|N|,|H|) = 1$ and $C_G (x) \le N$ for all $x \in N \setminus \{ 1 \}$.  In particular, every element of $G \setminus N$ acts fixed-point-freely on $N$.  One key property of Frobenius groups is that the conjugacy class of any element $g \in G \setminus N$ is a union of cosets of $N$.  This is the basis of the generalization of Camina pairs that have been studied widely in the literature.

Recall that the pair $(G,N)$ is called a {\it Camina pair} if $N$ is a proper, nontrivial normal subgroup of $G$ where every element $g \in G \setminus N$ has the property that its conjugacy class is a union of cosets of $N$.  It is not difficult to see that every Frobenius group with its Frobenius kernel yields a Camina pair.  In fact, it is possible to characterize the Camina pairs that are Frobenius groups.  If $(G,N)$ is a Camina pair, then $G$ is a Frobenius group with Frobenius kernel $N$ if and only if $(|G:N|,|N|) = 1$ and if and only if there is a complement $H$ for $N$ in $G$.

Camina pairs have been generalized further to Camina triples.   A {\it Camina triple} is a triple $(G,N,M)$ where $M \le N$ are nontrivial proper normal subgroups of $G$ so that every element $g \in G \setminus N$ satisfies the property that its conjugacy class is a union of cosets of $M$.  Camina triples motivate our generalization of Frobenius groups.  

To be precise, we study groups $G$ with two nontrivial proper normal subgroups $M$ and $N$, with $M\le N$, such that every element of $G$ lying outside of $N$ acts fixed-point-freely on $M$. Observe that if $M=N$ then $G$ is Frobenius group with Frobenius kernel $N$. We will call such a triple $(G,N,M)$ a {\it Frobenius triple}.  In our first theorem, we show that Frobenius triples play the same role for Camina triples that Frobenius groups play for Camina pairs.

\begin{introthm}\label{partconverse}
Let  $(G,N,M)$ be a Camina triple. The following are equivalent:
\begin{enumerate}[label={\bf(\arabic*)}]
\item $(G,N,M)$ is a Frobenius triple;
\item $(\norm{G:N},\norm{M})=1$;
\item There exists a subgroup $H\le G$ so that $G=HN$, $H\cap M=1$, and $(HM,(H \cap N)M,M)$ is a Camina triple.
\end{enumerate}
\end{introthm}

We note that the requirement in (3) that $(HM,(H \cap N)M,M)$ is a Camina triple is necessary.  Suppose that $p$ is an odd prime and that $G$ is an extra-special $p$-group of order $p^3$.  Take $M = Z(G)$ and let $N$ be a subgroup of order $p^2$, then $(G,N,Z(G))$ is a Camina triple.  Fix an element $x \in G \setminus N$, and let $H = \langle x \rangle$.  Then $G = HN$ and $H \cap Z(G) = 1$.  It is not difficult to see that $(G,N, Z(G))$ is not a Frobenius triple.  We thank Viji Thomas for this example. 

Although Frobenius triples and Frobenius groups are defined similarly, we will see that the structure of Frobenius triples is far less restricted. For example, we will see that there exist Frobenius triples $(G,N,M)$ where $M$ has arbitrarily large Fitting height, whereas Frobenius kernels are always nilpotent. We will also see that one has no control over the structure of the quotient $N/M$. In fact, any group can be a direct factor of $N/M$ when $(G,N,M)$ is a Frobenius triple.

In \cite{FWgps}, Wielandt studies finite groups $G$ possessing a subgroup $H$ for which every conjugate of $H$ intersects $H$ within some fixed normal subgroup $L$ of $H$. These triples $(G,H,L)$ have come to be known as Frobenius--Wielandt triples. We will show that Frobenius--Wielandt groups and Frobenius triples are intimately connected.

\begin{introthm}\label{split FT equiv conditions}
Let $M$ and $N$ be nontrivial, proper normal subgroups of $G$ satisfying $M\le N$. Let $H\le G$ satisfy $G=HM$ and $H\cap M=1$. Then $(G,N,M)$ is a Frobenius triple if and only if $(G,H,H\cap N)$ is a Frobenius--Wielandt triple.
\end{introthm}

Frobenius groups have a particularly nice structure from a representation theoretic perspective as every irreducible character of a Frobenius group $G$ with Frobenius kernel $N$ either has $N$ in its kernel or is induced from a nonprincipal (irreducible) character of $N$. In fact, this is a condition that characterizes the Frobenius kernel. Camina pairs and Camina triples have similar definitions in terms of certain characters inducing homogeneously. 

\begin{introthm}\label{introindcent}
Let $M$ and $N$ be nontrivial, proper normal subgroups of $G$ satisfying $M\le N$. If $(\norm{G:N},\norm{M})=1$, then $(G,N,M)$ is a Frobenius triple if and only if $\psi^G$ is irreducible for every character $\psi \in \irr (N \mid M)$.
\end{introthm}

As a straightforward application of Theorems~\ref{split FT equiv conditions} and \ref{introindcent}, we can describe the character theory of a Frobenius--Wielandt group $(G,H,L)$ in the situation that $H$ has a normal complement. 

\begin{introthm}\label{FW thm}
Let  $(G,H,L)$ be a Frobenius--Wielandt group with kernel $N$. Suppose that $M\lhd G$ satisfies $G=HM$ and $H\cap M=1$. Then $N=LM$, and $(G,N,M)$ is a Frobenius triple. In particular, $C_M(g)=1$ for every $g\in G\setminus N$ and $\psi^G\in\irr(G)$ for every $\psi\in\irr(N\mid M)$.
\end{introthm}

As far as we are aware, this is the only paper to address the character theory of Frobenius--Wielandt groups, albeit in a somewhat limited case. However, in the situation of Theorem~\ref{FW thm} we see that the Frobenius--Wielandt kernel enjoys similar properties held by Frobenius kernels. In some sense, the quotient $N/M$ can be considered as a measure of how close the Frobenius--Wielandt kernel is to being a Frobenius kernel in this situation, if $L$ is chosen as small as possible.

The Frobenius kernel of a Frobenius group is uniquely determined. It is unfortunately not the case that $N$ or $M$ is uniquely determined if $(G,N,M)$ is a Frobenius triple. We will see however that for a fixed $N$ there is a maximal choice of $M$ and for a fixed $M$ there is a minimal choice of $N$. A similar result is true for Camina triples. In \cite{SBMLnested}, the present authors construct for any normal subgroup of $G$ a pair of normal subgroups that can be used to characterize Camina triples. These constructions satisfy a compatibility condition on the lattice of normal subgroups of $G$ known as a (monotone) Galois connection. In Section~\ref{galois} we discuss similar constructions for Frobenius triples. Specifically, we define for any normal subgroup $N$ of $G$ two subgroups $C(G\mid N)$ and $I(G\mid N)$ that also characterize Frobenius triples: $(G,N,M)$ is a Frobenius triple if and only if $C(G\mid M)\le N$, which happens if and only if $M\le I(G\mid N)$. In particular, these constructions give another Galois connection on the lattice of normal subgroups of $G$.

\section{Frobenius and Camina triples}

Frobenius groups have many interesting properties.  One such property is that the conjugacy class of any element lying outside of the Frobenius kernel $N$ is a union of cosets of $N$.  This property has been generalized to obtain Camina pairs.  Let $G$ be a group and let $N$ be a proper, nontrivial normal subgroup of $G$.  We say $(G,N)$ is a {\it Camina pair} if the conjugacy class of every element of $G \setminus N$ is a union of $N$-cosets and the normal subgroup $N$ is said to be a {\it Camina kernel}.

The following result surveys some of the known results about Frobenius groups and their connection to Camina pairs.

\begin{thm}\label{equivFrob}
Let $N$ be a normal subgroup of $G$. The following are equivalent:
\begin{enumerate}[label={\bf(\arabic*)}]
\item The group $G$ is a Frobenius group with Frobenius kernel $N$.
\item Every element $1\ne x\in N$ satisfies $C_G(x)\le N$.
\item {\normalfont(\cite[Lemma 1]{acamina})}\hspace{\labelsep}
The pair $(G,N)$ is a Camina pair and $(\norm {G:N}, \norm {N}) = 1$.
\item {\normalfont(\cite[Proposition 3.2]{genFrobGps1})}\hspace{\labelsep}
The pair $(G,N)$ is a Camina pair and $G$ splits over $N$.
\end{enumerate}
\end{thm}

There are many conditions that are equivalent to the definition of Camina pairs.  We now give one.

\begin{lem}[{\normalfont \cite[Lemma 1]{acamina}}]\label{equivCamPr}
Let $N$ be a normal subgroup of $G$. Then $(G,N)$ is a Camina pair if and only if $\norm{C_G(g)}=\norm{C_{G/N}(gN)}$ for  every element $g\in G\setminus N$.
\end{lem}

For proofs of the above results, and other information about Camina pairs, we refer the reader to \cite{acamina,genFrobGps1,genFrobGps2,NM14}.

The idea of Camina pairs has been generalized even further.  Let $M \le N$ be proper, nontrivial normal subgroups of a group $G$.  We say $(G,N,M)$ is a {\it Camina triple} if the conjugacy class of every element of $G \setminus N$ is a union of $M$-cosets.  It is not difficult to see that $(G,N)$ is a Camina pair if and only if $(G,N,N)$ is a Camina triple.  Mattarei first considered these in his Ph. D. dissertation \cite{SM92}, although the term \enquote{Camina triple} never appears in  \cite{SM92}.  As far as we can tell, the term \enquote{Camina triple} first appears in \cite{weak}.  Camina triples are studied extensively by Mlaiki in \cite{NM14}, where he proves several results about Camina triples that are analogous to known results about Camina pairs. In fact, some of these are generalizations of results of the second author appearing in \cite{MLvos09}.  One such result is the following.

\begin{lem}[{\normalfont cf. \cite[Theorem 4]{NM14}}]\label{equivCam}
Let $M$ and $N$ be normal subgroups of $G$ satisfying $M\le N$. Then $(G,N,M)$ is a Camina triple if and only if $\norm {C_G(g)} = \norm{C_{G/M}(gM)}$ for every element $g\in G\setminus N$.
\end{lem}

In some ways, both $N$ and $M$ are {\it trying} to be a Camina kernel when $(G,N,M)$ is a Camina triple. Unfortunately, one cannot expect to say anything about the structure of the quotient group $N/M$ in general. In fact, it is not difficult to see that if $(G,N)$ is a Camina pair and $A$ is any group, then $(G\times A, N\times A, N)$ is a Camina triple. 

As mentioned above, a Frobenius group $G$ is a group possessing a nontrivial normal subgroup $N$ so that $C_G(x)\le N$ for each $1\ne x\in N$. When $(G,N)$ is a Camina pair, we already have $\norm{C_G(g)}=\norm{C_{G/N}(gN)}$ for every $g\in G\setminus N$. Thus a Camina pair $(G,N)$ is a Frobenius group exactly when the canonical homomorphism $C_G(g)\to C_{G/N}(gN)$ is an isomorphism. Applying this idea in the situation of Lemma~\ref{equivCam} leads to the following generalization of a Frobenius group.

Let $M$ and $N$ be nontrivial proper normal subgroups of $G$ with $M\le N$. If $C_G(x)\le N$ for every element $1\ne x\in M$, we will call $(G,N,M)$ a {\it Frobenius triple}. 

Using the idea preceding the definition, we show that Frobenius triples are, in fact, examples of Camina triples.

\begin{lem}\label{camfrob}
A Frobenius triple is a Camina triple.
\end{lem}

\begin{proof}
Let $(G,N,M)$ be a Frobenius triple and let $g\in G\setminus N$. Then $C_M(g)=1$, when means that the natural homomorphism $C_G(g)\to C_{G/M}(gM)$ is injective. Since $\norm{C_G(g)}\ge\norm{C_{G/M}(gM)}$ (e.g., see \cite[Corollary 2.24]{MI76}), we have $\norm{C_G(g)}=\norm{C_{G/M}(gM)}$. Thus $(G,N,M)$ is a Camina triple by Lemma~\ref{equivCam}.
\end{proof}

We know from Theorem~\ref{equivFrob} that a Camina pair $(G,N)$ is a Frobenius group if and only if $(\norm {G:N}, \norm {N}) = 1$. We now prove the equivalence of the first two statements of Theorem~\ref{partconverse}, which shows that a similar relationship exists between Camina pairs and Frobenius triples.

\begin{lem}\label{partconverse 1}
A Camina triple $(G,N,M)$ is a Frobenius triple if and only if $(\norm {G:N}, \norm {M}) = 1$.
\end{lem}

\begin{proof}
First assume that $(\norm{G:N},\norm{M})=1$. Write $\pi=\pi(M)$ and let $1\ne m\in M$. 	Let $g\in G\setminus N$. We may write $g=xy$ where $o(x)$ is a $\pi$-number and $o(y)$ is a $\pi'$-number and $xy=yx$.  Since $o(xN)$ divides $o(x)$ and $(\norm{G:N},\norm{M})=1$, it follows that $x\in N$ and $y\notin N$. Note that $M$ is solvable (see \cite[Lemma 2.7]{NM14}). Since $(G,N,M)$ is a Camina triple and $y\notin N$, we have $\norm{C_G(y)}=\norm{C_{G/M}(yM)}$. Since $(o(y),\norm{M})=1$ and $M$ is solvable, we see that $C_{G/M}(yM)=C_G(y)M/M$ \cite[Corollary 3.28]{MI08}. These two facts imply that $C_M(y)=1$. Since $y\in\inner{g}$, we deduce that $g\notin C_G(m)$ and so $C_G(m)\le N$. 

Now let $(G,N,M)$ be a Frobenius triple. Let $p$ be a prime dividing $\norm{M}$, and let $P\in\mathrm{Syl}_p(G)$. Then $P \cap M \in \mathrm{Syl}_p(M)$ and $P\cap M\lhd P$. Let $1\ne y\in M\cap Z(P)$. Then $P\le C_G(y)\le N$, so $p$ does not 	divide $\norm{G:N}$.
\end{proof}

 
\section{Connection to Frobenius--Wielandt triples}

In this section, we consider another generalization of Frobenius groups due to Wielandt.  First, we show that if we have a Frobenius triple, then we can find a \enquote{pseudo-complement.}  That is if $(G,N,M)$ is a Frobenius triple, then there exists a subgroup $H$ satisfying $G = HN$ and $H \cap M = 1$.  If $N = M$, then $H$ would be a complement.  Note that we only need the index of $N$ to be coprime to the order of $M$ for $H$ to exist.  Hence, this is really a generalization of the existence part of the Schur--Zassenhaus theorem.

\begin{lem}\label{psuedocomp}
Let $M$ and $N$ be normal subgroups of $G$ so that $(|G:N|,|M|) = 1$.  Then $M\le N$ and there exists a subgroup $H \le G$ so that $G=HN$ and $H\cap M=1$.  Moreover, if $(G,N,M)$ is a Frobenius triple, then $(HM,(H\cap N)M,M)$ is also a Frobenius triple.
\end{lem}

\begin{proof}
The first statement is clear since $NM/N\cong M/(M\cap N)$ and $(\norm{G:N},\norm{M})=1$. We prove the second statement by induction on $\norm{G}$. First we claim that there exists a maximal subgroup of $G$ not containing $N$. To see this, suppose on the contrary that $N\le\Phi(G)$. Then every prime divisor of $\norm{M}$ divides $\norm{\Phi(G)}$ and does not divide $\norm{G:\Phi(G)}$, which is impossible. Thus we may find a maximal subgroup $K$ not containing $N$, as claimed. If $M\cap K=1$, then we are done. So we assume that $M\cap K>1$. Then $\norm{G:N}=\norm{K:K\cap N}$ is relatively prime to $\norm{K\cap M}$, so by the inductive hypothesis there exists $H\le K$ so that $K=H(K\cap N)$ and $H\cap M=1$. Thus $G=NK=NH$ and $H\cap M=1$, as desired. 
	
Suppose $(G,N,M)$ is a Frobenius triple; so we have $C_{HM}(x)\le HM\cap N=(H\cap N)M$ for every element $1\ne x\in M$. Thus $(HM,(H\cap N)M,M)$ is a Frobenius triple.
\end{proof}

When $M$ and $N$ are normal subgroups of $G$ so that $(\norm{G:N},\norm{M})=1$, we call the subgroup $H$ of Lemma~\ref{psuedocomp} a {\it pseudo-complement} for $(G,N,M)$.

In \cite{FWgps}, Wielandt studies groups $G$ that have a nontrivial, proper subgroup $H$ and a normal subgroup $L$ of $H$ that is proper in $H$ for which $H^g \cap H \le L$ for every element $g \in G\setminus H$. In this situation, Espuelas \cite{espuelas} says that $(G,H,L)$ is a Frobenius-Weilandt group. To be consistent with our notation, we will instead say that $(G,H,L)$ is a {\it Frobenius-Wielandt triple}. If $(G,H,L)$ is a Frobenius--Wielandt triple, there is a unique normal subgroup $N$ of $G$ satisfying $G = NH$ and $N \cap H=L$. This subgroup is given by $N = G \setminus \bigcup_{g\in G} (H\setminus L)^g$  and is called the {\it Frobenius--Wielandt kernel}. Note that the uniqueness of $N$ depends on $L$.  On the other hand, observe that $L$ is not uniquely defined; one can easily replace $L$ by any normal subgroup of $H$ that contains $L$. For this reason Espuelas defines the {\it Frobenius-Wielandt complement} to be the quotient $H/L$.

Our next result shows that Frobenius triples and Frobenius--Wielandt groups are intimately connected.

\begin{thm}\label{FT FW equiv}
Let $M \le N$ be proper, nontrivial normal subgroups of $G$ and assume that $H$ is a pseudo-complement for $(G,N,M)$. Then $(HM,(H \cap N)M,M)$ is a Frobenius triple if and only if $(HM,H,H \cap N)$ is a Frobenius--Wielandt triple. 
\end{thm}

\begin{proof}
First, suppose that $(HM,H,H \cap N)$ is a Frobenius--Wielandt triple; so that $H^x\cap H\le H\cap N$ for every element $1\ne x\in M$. Observe that this implies that $C_H(x)\le H\cap N$ for every $1\ne x\in M$. Let $K$ be the Frobenius--Wielandt kernel. Then $K=HM\setminus\bigcup_{x\in M}(H\setminus N)^x$ and $K$ satisfies $HM=HK$ and $H\cap K=H\cap N$. Since $M\le N$, it follows that $M\le K$. Also, $\norm {HM:K} = \norm {HK:K} = \norm {H:H\cap K} = \norm {H:H\cap N} = \norm{HM:(H\cap N)M}$, which gives $K=(H\cap N)M$. Thus, we see that $HM\setminus N=HM\setminus (H\cap N)M$ is partitioned by the sets $(H\setminus N)^x$ for $x\in M$. Let $g\in HM\setminus N$, let $1\ne x\in M$ and assume that $[g,x]=1$. Then there exists $y\in M$ so that $g\in (H\setminus N)^y$. So $g=h^y$ for some $h\in H\setminus N$ and it follows that $(x^{y^{-1}})^h=x^{y^{-1}}$. Thus $h\in C_H(x^{y^{-1}})\le H\cap N$. But $h\in H\setminus N$, a contradiction. Thus, no such element $g$ exists and so $C_{HM}(x)\le (H\cap N)M$. It follows that $(HM,(H\cap N)M,M)$ is a Frobenius triple.

Now, assume that $(HM,(H\cap N)M,M)$ is a Frobenius triple.  This implies that $H \cap N$ will be a proper, normal subgroup of $H$.  Consider elements $g\in HM$ and $h\in H^g\cap H$. We see that $H^g\cap H=H^x\cap H$ for some element $x\in M$ and so $h=k^x$ for some $k\in H$. Then $[k,x] = k^{-1}k^x = k^{-1}h \in H \cap M = 1$, so $k = h = k^x$. This implies that $x \in C_M(k)$, and so, $h = k \in H \cap N$. Hence $HM$ is a Frobenius--Wielandt group with Frobenius-Wielandt complement $H/(H\cap N)$, as desired.
\end{proof}

As a consequence of Theorem~\ref{FT FW equiv}, we see that every Frobenius triple contains a subgroup that is a Frobenius--Wielandt triple. Let $(G,H,L)$ be a Frobenius--Wielandt triple and suppose that $M$ is a normal subgroup of $G$ satisfying $H\cap M=1$. Then $(HM,H,L)$ is also a Frobenius--Wielandt triple. Let $N\le HM$ be the Frobenius--Wielandt kernel. Then $H\cap N=L$ and so $(HM,(H\cap N)M,M)$ is Frobenius triple by Theorem~\ref{FT FW equiv}. So at least in this situation, a Frobenius--Wielandt triple contains a subgroup that is a Frobenius triple. This is not always possible however. As an example, let $G$ be the general linear group $\mathrm{GL}_2(\mathbb{F}_3)$, let $H\in\mathrm{Syl}_2(G)$, and let $N$ be the special linear group $\mathrm{SL}_2(\mathbb{F}_3)$. Then $H\cap H^x\le H\cap N$ for every $x\in G\setminus N$ and so $(G,H,H\cap N)$ is a Frobenius--Wielandt triple with kernel $N$. However, $Z(G)$ is the unique minimal normal subgroup of $G$ and $Z (G) \le H$.  Hence, there is no normal subgroup $M$ of $G$ with $H \cap M = 1$. 

Under the added assumption that $(G,N,M)$ is a Camina triple, we obtain a result slightly stronger than Theorem~\ref{FT FW equiv}.

\begin{lem}\label{reduction}
Let $(G,N,M)$ be a Camina triple. If $H$ is a pseudo-complement for $(G,N,M)$, then $(G,N,M)$ is a Frobenius triple if and only if $(HM,(H\cap N)M,M)$ is a Frobenius triple.
\end{lem}

\begin{proof}
When $(G,N,M)$ is a Frobenius triple, the fact that $(HM,(H\cap N)M,M)$ is a Frobenius triple follows from Lemma~\ref{psuedocomp}. 

Conversely, suppose $(HM,(H\cap N)M,M)$ is a Frobenius triple.  It follows from Lemma~\ref{partconverse 1} that $(|HM:(H \cap N)M|,|M|) = 1$.  Since $\norm{HM:(H\cap N)M}=\norm{H:H\cap N}=\norm{G:N}$, we have $(|G:N|,|M|) = 1$, and in light of Lemma \ref{partconverse 1} again, we conclude that $(G,N,M)$ is a Frobenius triple.
\end{proof}

Proposition~3.2 of \cite{genFrobGps1} shows that $G$ is a Frobenius group with Frobenius kernel $N$ if and only if $(G,N)$ is a Camina pair and $G$ splits over $N$. Our next result is a generalization of this fact for Camina triples, and its proof completes the proof of Theorem~\ref{partconverse}.   

We note that in general if $H$ is a psuedo-complement, $(G,N,M)$ being a Camina triple does not imply that $(HN,(H \cap N)M,M)$ is necessarily a Camina triple.  Observe that if $p$ is an odd prime, and $G$ is an extra-special $p$-group of order $p^3$ and exponent $p$, then taking $N$ to be a supgroup of order $p^2$ and $H$ a subgroup of order $p$ not contained in $N$, then as we have seen $(G,N,Z(G))$ is a Camina triple with $G = HN$ and $H \cap Z(G) = 1$, but $(HN,Z(G),Z(G)) = (HN, (H \cap N)Z(G),Z(G))$ is not a Camina triple.

\begin{thm}\label{cam psuedosplit}
Let $(G,N,M)$ be a Camina triple. Then $(G,N,M)$ is a Frobenius triple if and only if there exists a subgroup $H \le G$ so that $G = HN$, $H \cap M = 1$, and $(HM,(H \cap N)M,M)$ is a Camina triple. 
\end{thm}

\begin{proof}
If $(G,N,M)$ is a Frobenius triple, then there exists a subgroup $H\le G$ so that $G = NH$ and $H \cap M = 1$ by Lemma~\ref{psuedocomp}.  By Lemma \ref{reduction}, we see that $(HM,(H\cap N)M,M)$ is a Frobenius triple.  Applying Lemma \ref{camfrob}, we deduce that $(HM,(H\cap N)M,M)$ is a Camina triple.

Conversely, assume that there exists a subgroup $H \le G$ so that $G = HN$, $H \cap M=1$, and $(HM,(H\cap N)M,M)$ is a Camina triple.  Using Lemma~\ref{reduction}, we see that to prove that the triple $(G,N,M)$ is a Frobenius triple, it suffices to prove that the Camina triple $(HM,(H \cap N)M, M)$ is a Frobenius triple.  Thus, we may assume that $G = HM$ and $N = (H  \cap N)M$.  Note that we are using the assumption that $(HM,(H \cap N)M,M)$ is a Camina triple to make this replacement.  

Fix the elements $g \in H \setminus N$ and $z \in M$. Since $(G,N,M)$ is a Camina triple, there exists an element $y \in G$ so that $[g,y] = z$ (see \cite[Theorem 2.1 (v)]{NM14}).  Write $y = hm$, where $h \in H$ and $m \in M$. Then $[g,y] = [g,m] [g,h]^m$.  Since $g, h \in H$, we have $[g,h] \in H$.  On the other hand, we know $[g,y] \in M$ and as $m \in M$, we have $[g,m] \in M$.  We determine that $[g,h]^m = [g,m]^{-1} [g,y] \in M$, and so it follows that $[g,h] \in H \cap M=1$. Thus $[g,m]=z$. This means that $\{ [g,m] \mid m \in M \} = M$, and so, $g$ acts fixed-point-freely on $M$. Consider elements $1\ne x\in M$ and $h\in H^x\cap H$. Then there exists an element $k \in H$ so that $h = k^x$, and so, $k^{-1}h = k^{-1}k^x = [k,x] \in H \cap M=1$. Hence $k \in C_H(x)\le H\cap N$, from which it follows that $H^x \cap H\le H\cap N$. Thus $(G,H,H\cap N)$ is a Frobenius--Wielandt triple, and therefore, the conclusion that $(G,N,M)$ is a Frobenius triple follows from Theorem~\ref{FT FW equiv}.  
\end{proof}

We now prove Theorem~\ref{split FT equiv conditions}.  Notice that condition (1) is the hypothesis that $(G,N,M)$ is a Frobenius triple.  On the other hand, condition (4) is the condition that $(G,H,H \cap N)$ is a Frobenius-Wielandt triple.  

\begin{thm}\label{equiv}
Let $M$ and $N$ be proper, nontrivial normal subgroups of $G$ that satisfy $M \le N$, and assume that $M$ has a complement $H$ in $G$. The following are equivalent:
\begin{enumerate}[label={\bf(\arabic*)}]
\item $C_G(x) \le N$ for every element $1\ne x\in M$;
\item $C_H(x) \le H \cap N$ for every element $1\ne x\in M$;
\item $C_G(h) \le H$ for every element $h\in H \setminus N$;
\item $H \cap N < H$ and $H^g \cap H \le H \cap N$ for every element $g\in G \setminus H$.
\item Every element $g \in G\setminus N$ is conjugate to an element of $H$; 
\item If $h \in H \setminus N$, then $h$ is conjugate to every element of $hM$.
\end{enumerate}
\end{thm}

\begin{proof}
The observation that condition (1) implies condition (2) is obvious.  Assume condition (2) now.  Consider an element $h\in H\setminus N$ and suppose that $g=km\in C_G(h)$ where $k\in H$ and $m\in M$. Then $1=[km,h]=[k,h]^m[m,h]$ and so $[k,h]\in H\cap M=1$. Thus $[m,h]=1$ and so $m=1$ and $g\in H$. Hence (3) holds.

Now assume (3).  Since $G = HM$, we have $G = NH$ and $|G:N| = |H:H \cap N|$.  Because $N$ is proper in $G$, we have that $H \cap N$ is proper in $H$.  Consider an element $g \in G \setminus H$.  There exists elements $y \in H$ and $1\ne x\in M$ so that $g = yx$.  Observe that $H^g = H^{yx} = H^x$.  Suppose $h \in H^g \cap H = H^x\cap H$. Then there exists $k\in H$ so that $h=k^x$ and so $k^{-1}h = [k,x] \in H \cap M = 1$. So $h = k \in C_H(x) \le H\cap N$ and (4) follows.

Assume that (4) holds. Thus, $(HM,H,H\cap N)$ is a Frobenius--Wielandt group with kernel $K = HM \setminus \left[ \bigcup_{x \in M} (H^x \setminus N) \right]$.  Since $M \le N$, we see that $M \le K$. Since $HM = HK$ and $H \cap K = H \cap N$, we deduce that $\norm{HM:K} = \norm{HK:K} = \norm{H:H\cap N} = \norm{HM:(H\cap N)M}$ and so $K = (H\cap N)M = HM \cap N = N$. By the description of $K = HM \cap N$ given above, it is clear that every element of $HM$ lying outside of $N$ is contained in one of the sets $(H \setminus N)^x$ for some element $x \in M$. Thus (5) is proved.

Given (5), let $h\in H \setminus N$ and $m \in M$. We wish to show that $h$ is conjugate to $hm$. By (5), some conjugate of $hm$ lies in $H$, say $(hm)^g = h_1 \in H$ for $g \in G$.  We can write $g = kl$ where $k \in H$ and $l \in M$.  Then $h_1 = (hm)^{kl} = (kl)^{-1}(hm)(kl) = l^{-1}k^{-1}hmkl$.  We know $l^{-1} k^{-1}h \in M k^{-1}h = k^{-1}h M$, so $l^{-1} k^{-1}h = k^{-1}h l^*$ for some $l^* \in M$.  Then $l^*mk \in Mk = kM$ and so, $l^*mk = k m^*$ for some $m^* \in M$.  We deduce that 
$h_1  = l^{-1}k^{-1}hmkl = k^{-1}hk m^*l$.  This implies that $h_1^{-1} k^{-1} hk = (m^*l)^{-1}$.  This implies that $h_1^{-1} k^{-1} hk \in H \cap M = 1$.  We obtain $h_1 = h^k$.  Since $h_1 = (hm)^g$, we conclude that $(hm)^{gk^{-1}}=h$, which proves (6). 

Finally assume (6). We will show that $(G,N,M)$ is a Camina triple. To do this we will show that every element $g \in G \setminus N$ is conjugate to every element in $gM$. Write $g=hm$, where $h \in H\setminus N$ and $m \in M$. Let $x \in M$. By (6), we can find $y \in G$ so that $g^y = (hm)^y = h$ and $z \in G$ so that $h = (hx)^z$. So $(hm)^{yz^{-1}} = hx$, as required. So $(G,N,M)$ is a Camina triple.  Since there exists a pseudo-complement $H$, we may use Theorem~\ref{cam psuedosplit} to see that $(G,N,M)$ is a Frobenius triple, and thus, condition (1) follows.
\end{proof}

Note that the assumption $G=HM$ is necessary in the statement of Theorem~\ref{split FT equiv conditions}. To see this, let $G=C_3\times S_3$. Let $M$ be a Sylow 3-subgroup of $S_3$, let $H$ be a Sylow $2$-subgroup of $S_3$ and let $N$ be a Sylow $3$-subgroup of $G$. Then $(G,N,M)$ is a Frobenius triple, $G=HN$ and $H\cap M=1$. However, conditions (3)--(6) of Theorem~\ref{split FT equiv conditions} do not hold.

\section {Properties of Frobenius triples}

We now study Frobenius triples and we determine some of their properties.  We first show that if we have a Frobenius triple where the ``pseudo-complement'' actually complements the lower term in the triple, then in fact, we have a Frobenius group.

\begin{lem}\label{N splits}
Let $(G,N,M)$ be a Frobenius triple. If there exists a subgroup $H \le G$ that satisfies $H \cap N = 1$, then $MH$ is a Frobenius group with Frobenius kernel $M$. In particular, $M$ is nilpotent.
\end{lem}

\begin{proof}
Suppose $1 \ne m \in M$, then $C_{G} (m) \le N$.  It follows that $C_{MH} (m) = C_{G} (m) \cap MH \le N \cap MH = M$.  By Theorem \ref{equivFrob}, this implies $HM$ is a Frobenius group.  
\end{proof}

In the situation of Lemma~\ref{N splits}, we can find an element of prime order in $G\setminus N$ and, of course, $M$ is nilpotent. Using a similar argument as in Lemma~4 of \cite{acamina}, we see that having an element of prime order in $G\setminus N$ is always enough to guarantee that $M$ is nilpotent.  Note that if $(G,N,M)$ is a Frobenius triple and $G \setminus N$ contains an element of prime order, then this next lemma is even easier to prove.

\begin{lem}\label{nilpotentCamina}
Let $(G,N,M)$ be a Camina triple. If there is an element $x\in G\setminus N$ of prime order, then $M$ is nilpotent. 
\end{lem}

\begin{proof}
Let $g\in M$ and let $x\in G\setminus N$ with $o(x)=p$. Since $(G,N,M)$ is a Camina triple, $x^{-1}$ is conjugate to $x^{-1}g$, which is conjugate to $gx^{-1}$. In 	particular, $(gx^{-1})^p=1$. Rewriting this, we have $gg^xg^{x^2}\dotsb g^{x^{p-1}}=x^p=1$. The result now follows from a result of Hughes, Kegel and Thompson (see \cite[V 8.13 Hauptsatz]{huppertendliche}).
\end{proof}
 
In general, $M$ is not nilpotent when $(G,N,M)$ is a Frobenius triple. In fact, one may find examples of Frobenius triples $(G,N,M)$ where the Fitting height $h(M)$ of $M$ is arbitrarily large.  This next theorem contains most of the work of proving these examples exist.
 
\begin{thm}\label{Fitting height}
For every positive integer $n$ and every prime $p$, there exists a group $G$ that satisfies the following conditions:
\begin{enumerate}[label={\rm(\roman*)}]
\item $G$ is a semidirect product a group $M$ by a group $H$.
\item $H$ is a cyclic group  of order $p^n$
\item $H$ acts faithfully on $M$.
\item $M$ is a solvable $p'$-group whose Fitting height is $n$.
\item The Fitting subgroup of $M$ is minimal normal in $G$.
\item For every nonidentity element $x \in M$, $p$ divides $\norm{G:C_G (x)}$.
\end{enumerate}
\end{thm}
 
\begin{proof}
We work by induction on $n$.  When $n = 1$, we take $G$ be a Frobenius group whose Frobenius complement has order $p$ and whose Frobenius kernel $M$ is minimal normal in $G$.  Note that $G$ will satisfy the stated conditions.

We now assume $n \ge 1$, and we assume that we have an example $G$ that satisfies the given conditions, and we prove that there exists a group that satisfies the conditions for $n+1$.  Thus, $G$ is the semi-direct product of a cyclic group $H$ of order $p^n$ acting faithfully on the solvable $p'$-group $M$ whose Fitting height is $n$ such that the Fitting subgroup of $M$ is minimal normal in $G$ and every non identity element $x \in M$ satisfies $\norm {H:C_H (x)}$ is divisible by $p$.  We can write $H = \langle h \rangle$.  Take $\tilde H$ to be a cyclic group of order $p^{n+1}$ and write $\tilde H = \langle \tilde h \rangle$.  We can define an action of $\tilde H$ on $M$ by $m^{\tilde h} = m^h$ for all $m \in M$, and note that $Z = \langle \tilde h^{p^n} \rangle$ will be the kernel of this action.  We now define $\Gamma$ to be the semi-direct product of $\tilde H$ acting on $M$.  Observe that $Z$ will be the center of $\Gamma$ and $Z \times M$ will be the Fitting subgroup of $\Gamma$.  

Take $\mu \in \irr (M)$ and $\lambda \in \irr (Z)$ to be nonprincipal characters, and take $\chi$ to be an irreducible constituent of $(\mu \times \lambda)^G$.  It is not difficult to see that $\chi$ will have to be faithful.  Let $q$ be a prime that does not divide $|G|$.  Thus, we may view $\chi$ as an irreducible $q$-Brauer character.  Now, choose $V$ to be an irreducible $Z_q$-module that affords $\chi$ as a $q$-Brauer character.  We now have an action of $\Gamma$ on $V$.  It is not difficult to see that $\Gamma$ will act faithfully on $V$.  Observe that $C_V (Z)$ will be a proper $G$-submodule of $V$, and so, $C_V (Z) = 0$.  

We now define $\tilde G$ to be the semi-direct product of $\Gamma$ acting on $V$, and we define $\tilde M$ to be $MV$, so $\tilde M$ is the semi-direct product of $M$ acting on $V$.  It follows that $\tilde M$ is a solvable $p'$-subgroup.  Note that $\tilde G$ will be the semi-direct product of $\tilde H$ acting faithfully on $\tilde M$.  Observe that $V$ will be the Fitting subgroup of $\tilde M$, and since $M$ has Fitting height $n$, this implies that $\tilde M$ has Fitting height $n+1$.  Since $V$ is the Fitting subgroup of $\tilde M$, it follows that the Fitting subgroup of $\tilde M$ is minimal normal in $G$.  Finally, suppose $x$ is a nonidentity element of $\tilde M$.  Suppose first that $x \not\in V$.  Then $C_{\tilde G} (x) ZV/ZV \le C_{\tilde G/VZ} (xVZ)$.  Observe that $\tilde G/VZ$ is isomorphic to $G$ and $\tilde H/Z$ is isomorphic to $H$, and so, we have $p$ divides $|\tilde G/VZ:C_{\tilde G/VZ} (xVZ)|$.  In this case, we will have $p$ divides $\norm{\tilde G:C_{\tilde G} (x)}$.  Suppose $x \in V$.  Since $C_V (Z) = 0$, we see that $Z$ is not contained in $C_{\tilde G} (x)$.  On the other hand, we know that the Sylow $p$-subgroups of $\tilde G$ are cyclic; so this implies that a Sylow subgroup of $C_{\tilde G} (x)$ will be trivial.  This implies that $p^n$ divides $|\tilde G:C_{\tilde G} (x)|$.  We have now proven that $\tilde G$ with $\tilde H$ and $\tilde M$ satisfy the conclusions of the theorem for $n+1$.
\end{proof} 
 
We now show the existence of the stated examples.

\begin{cor}
Let $n$ be a positive integer. There is a Frobenius triple $(G,N,M)$ where $h(M)=n$.
\end{cor}
 
\begin{proof}
Take $G$, $H$ and $M$ to be as in Theorem~\ref{Fitting height}. Let $D$ be the subgroup of $H$ of index $p$, and let $N=DM$. Then $C_H(x)\le D=H\cap N$ for every $1\ne x\in M$ and so $(G,N,M)$ is a Frobenius triple by Theorem~\ref{split FT equiv conditions}.
\end{proof}
 
Although $M$ can have arbitrarily large Fitting height when $(G,N,M)$ is a Frobenius triple, we can bound the Fitting height of $M$ in terms of the order of elements in $G\setminus N$ by appealing to results of F. Gross and E. I. Khukhro regarding fixed-point-free automorphisms. 

\begin{thm}
Let $M$ be a nontrivial, proper normal subgroup of $G$ and suppose that $C = \inner{C_G(x):1\ne x\in M}$ is a proper subgroup of $G$. For each prime $p\in\pi(G/C)$, write $p^{n_p}$ for the minimum order of a $p$-element in $G\setminus C$. 
Then $h (M) \le \min\{n_p:p\in\pi(G/C) \}$.
\end{thm}

\begin{proof}
Since every $g\in G\setminus C$ acts fixed-point-freely on $M$, this follows from Theorem 1 of \cite{fpfautos_odd} when $p$ is odd and the Theorem in \cite{khukhro} which improved the earlier result of Theorem 1.1 of \cite{fpfautos_even} when $p = 2$.
\end{proof}

\section{Character theory of Frobenius triples}

It is well-known that $G$ is a Frobenius group with Frobenius kernel $N$ if and only if the induced character $\psi^G$ is irreducible for every nonprincipal character $\psi \in \irr(N)$.  
A Camina triple can also be characterized by a condition on its irreducible characters. Since these conditions will be used repeatedly throughout this section, we list them here. For a normal subgroup $M$ of $G$, we use $\irr(G\mid M)$ to denote the set of all irreducible characters of $G$ that do not contain $M$ in their kernel.

\begin{lem}[{\normalfont cf. \cite[Theorem 4]{NM14}}]\label{equivCam2}
Let $M$ and $N$ be proper, nontrivial normal subgroups of $G$ that satisfy $M \le N$. The following are equivalent:
\begin{enumerate}[label={\bf(\arabic*)}]
\item $(G,N,M)$ is a Camina triple.
\item Every character $\psi \in \irr (N \mid M)$ induces homogeneously to $G$.
\item Every character $\psi \in \irr (G \mid M)$ vanishes on $G \setminus N$.
\end{enumerate}
\end{lem}

The following result is a known consequence of Clifford's Theorem.  It (essentially) appears as Lemma 2.2 of \cite{gagola_chars}. We include a short proof here.

\begin{lem} \label{induce}
Let $N$ be a normal subgroup of $G$, and suppose that the character $\chi \in \irr (G)$ vanishes on $G \setminus N$. Let $\psi$ be an irreducible constituent of $\chi_N$. Then $\psi^G = \norm {T:N}^{1/2}\chi$, where $T=I_G(\psi)$ is the stabilizer of $\psi$ in $G$. In particular, $\psi$ is fully-ramified with respect to $T/N$.
\end{lem}

\begin{proof}
Write $e = \inner {\psi^G,\chi}$.  Consider the character $\xi \in \irr(G)$. Then
\begin{align*}
\inner{\psi^G,\xi}
			&=\frac{1}{\norm{G}}\sum_{g\in G}\psi^G(g)\overline{\xi(g)}=\frac{1}{\norm{G}}\sum_{g\in N}\psi^G(g)
				\overline{\xi(g)}\\
			&=\frac{1}{\norm{G}}\sum_{g\in N}\frac{1}{\norm{N}}\sum_{x\in G}\psi^x(g)\overline{\xi(g)}=
				\frac{1}{\norm{G}}\sum_{g\in N}\frac{\norm{T}}{e\norm{N}}\chi(g)\overline{\xi(g)}\\
			&=\frac{\norm{T:N}}{e}\cdot\frac{1}{\norm{G}}\sum_{g\in G}\chi(g)\overline{\xi(g)}=
				\frac{\norm{T:N}}{e}\,\inner{\chi,\xi}=\frac{\norm{T:N}}{e}\,\delta_{\chi,\xi}.
	\end{align*}	
It follows that $\psi^G=e\chi=\frac{\norm{T:N}}{e}\chi$ and so also that $\psi^G=\norm{T:N}^{1/2}\chi$.
\end{proof}

We now show that the irreducible characters of Frobenius triples satisfy a condition that is similar to the one satisfied by Frobenius groups.

\begin{thm}\label{Frobish}
If $(G,N,M)$ is a Frobenius triple, then $\psi^G$ is irreducible for every character $\psi\in\irr(N\mid M)$.
\end{thm}

\begin{proof}
We work by induction on $\norm{G}$.  Consider a character $\psi \in \irr (N \mid M)$. It is not difficult to see from Lemma~\ref{induce} and \cite[Theorem 6.11]{MI76} that $\psi^G \in \irr (G)$ if and only if $N$ is the stabilizer of $\psi$. Using the inductive hypothesis, we have $I_G(\psi) \cap K = I_K (\psi) = N$ for any nontrivial proper subgroup $K/N$ of $G/N$ and so it follows that $I_G(\psi)=N$ if $G/N$ has composite order. Thus it suffices to assume that $G/N$ has prime order. Let $\chi\in\irr(G\mid \psi)$. Since $\chi$ vanishes on $G\setminus N$ by Lemma~\ref{equivCam2}, we have $\psi^G = \norm {I_G (\psi) : N}^{1/2} \chi$ by Lemma~\ref{induce}. Since $G/N$ has prime order, we must have $I_G (\psi) = N$ and so $\psi^G$ is irreducible.
\end{proof}

When $M=N$, we recover the well-known fact that every nonprincipal irreducible character of a Frobenius kernel $N$ induces irreducibly to $G$. A standard proof of this fact is a straightforward application of Brauer's Permutation Lemma (see \cite[Theorem 6.32]{MI76}, for example). We can actually prove Theorem~\ref{Frobish} in a similar manner by using a variation of this result, which will also be used to obtain another character-theoretic characterization of Frobenius triples. Our proof is essentially the same argument as the one used in the proof of \cite[Theorem 6.32]{MI76}. 

\begin{lem}\label{newBrauer}
Let $N$ be a normal subgroup of $G$. Set $\mathrm {Cl}_G (N) = \{ \mathrm{cl}_G (n) \mid n \in N\}$ and $\irr_G (N) = \{ \mathrm{orb}_G (\psi) \mid \psi \in \irr(N) \}$. Suppose that the group $A$ acts on $G$ via automorphisms. If $N$ is $A$-invariant, then $A$ acts on $\mathrm{Cl}_G(N)$ and $\irr_G(N)$, and each $a\in A$ fixes the same number of elements in $\mathrm {Cl}_G (N)$ and $\irr_G (N)$.
\end{lem}

\begin{proof}
Since $A$ acts on $G$ by automorphisms, it is clear that $A$ acts on $\mathrm{Cl}_G(N)$. Also for every $\psi\in\irr(N)$, $a\in A$ and $g\in G$, we have $(\psi^g)^a=(\psi^a)^{g^a}$ and so $a$ maps $\mathrm{orb}_G(\psi)$ to $\mathrm {orb}_G (\psi^a)$. Note that for each $\kappa \in \mathrm {Cl}_G (N)$ and $\lambda \in \irr_G(N)$, the character $\chi^{}_\lambda = \sum_{\xi \in \lambda} \xi$ is constant on the elements of $\kappa$.  For each $\kappa \in \mathrm{Cl}_G(N)$, let $g_\kappa\in\kappa$ and let $X$ be the matrix $(\chi^{}_\lambda(g_\kappa))$. For each $\lambda \in \irr_G(N)$ and $a\in A$, define $\chi_\lambda^a=\chi_{\lambda^a}$ and observe that $\chi_\lambda^a(n^a)=\chi^{}_\lambda(n)$ for every $\lambda \in \irr_G (N)$, $a\in A$ and $n\in N$.  Let $P$ and $Q$ denote the permutation representations associated to the actions of $A$ on $\irr_G(N)$ and $\mathrm{Cl}_G(N)$, respectively. Then the $(\lambda,\kappa)$-entry of $P(a)X$ is $\sum_\mu P(a)_{\lambda\mu} \chi_\mu (g_\kappa) = \chi_\lambda^a (g_\kappa),$ and the $(\lambda,\kappa)$-entry of $XQ(a)$ is 	$\sum_\gamma\chi^{}_\lambda(g_\gamma)Q(a)_{\gamma\kappa}=\chi^{}_\lambda(g_\kappa^{a^{-1}})	= \chi_\lambda^a(g_\kappa).$  Thus, we see that $P(a)X = XQ(a)$. 

To complete the proof, it suffices to show that $X$ is invertible since then $P(a)$ and $Q(a)$ would have the same trace. To see this, first note that $X$ must be square by Brauer's Permutation Lemma. Write $\mathrm {Cl}_G (N) = \{\kappa_1, \dotsc, \kappa_n \}$ and $\irr_G (N) = \{\lambda_1,\dotsc,\lambda_n\}$. Define the matrices $T = (\delta_{ij} \norm{\lambda_i})$ and $D = (\delta_{ij} \norm{\kappa_i})$, where (as usual) $\delta_{ij}$ denotes the Kronecker $\delta$ function. Let $X^\dagger$ denote the conjugate transpose of $X$. Then $XDX^\dagger = \norm{G}T$ since $\inner {\chi^{}_{\lambda_i},\chi^{}_{\lambda_j}} = \delta_{ij} \norm{\lambda_i}$.  Therefore, we conclude that $\det(X) \ne 0$ and so $X$ is invertible, as desired.  
\end{proof}

We can also characterize Frobenius triples in terms of the stabilizers of nonprincipal characters.

\begin{lem}\label{inertial}
Let $M \le N$ be proper, nontrivial normal subgroups of $G$ satisfying $M\le N$. Then $(G,N,M)$ is a Frobenius triple if and only if the stabilizer of every nonprincipal character in $\irr (M)$ is contained in $N$.
\end{lem}

\begin{proof}
First assume that $(G,N,M)$ is a Frobenius triple. Fix $\theta$ to be a nonprincipal irreducible character of $M$, and consider an element $g \in G$ that stabilizes $\theta$. Note then that $g$ fixes the $N$-orbit in $\irr_N (M)$ containing $\theta$. From Lemma~\ref{newBrauer}, it follows that $g$ fixes a nonidentity $N$-class in $\mathrm {Cl}_N (M)$, say the $N$-conjugacy class of $x\in M$.  Then there exists an element $n\in N$ so that $x^g=x^n$, which yields $gn^{-1}\in C_G(x) \le N$. It follows that $g\in N$.  The proof of the reverse implication is similar.
\end{proof}

We now present an alternate proof of Theorem~\ref{Frobish} using Lemma~\ref{newBrauer}.

\begin{proof}[Alternate proof of Theorem~$\ref{Frobish}$]
Let $g\in G\setminus N$. Note that $g$ cannot fix the $N$-conjugacy class of any nonidentity element of $M$. For if it did, say the $N$-class of $x\in M$, then there would exist $n\in N$ so that $x^g=x^n$ and so $gn^{-1}\in C_G(x)$. But this cannot happen since $C_G(x)\le N$ and $g\notin N$. Since $g$ fixes no $N$-classes of nonidentity elements of $M$, $g$ fixes no $N$-orbits of nonprincipal irreducible characters of $M$ by Lemma~\ref{newBrauer}. By Clifford's Theorem, we deduce that $\psi^g_M\ne \psi_M$ for any $\psi\in\irr(N\mid M)$, and hence $\psi^g\ne \psi$ for any $\psi\in\irr(N\mid M)$. It follows that $\psi^G\in\irr(G)$ for each character $\psi \in \irr(N \mid M)$.
\end{proof}

Theorem~\ref{Frobish} and Lemma~\ref{newBrauer} yield the following result about the the action of a group on the set of conjugacy classes of a normal subgroup.

\begin{cor}
Let $(G,N,M)$ be a Frobenius triple. Then every element $g\in G\setminus N$ fixes exactly as many $N$-conjugacy classes in $N\setminus M$ as $gM$ fixes nonidentity conjugacy classes of $N/M$.
\end{cor}

\begin{proof}
By Brauer's Permutation Lemma \cite[Theorem 6.32]{MI76}, the number of $N$-conjugacy classes fixed by $g$ is equal to the number of irreducible characters of $N$ fixed by $g$. Since $g$ fixes no nonidentity classes of $M$ and no elements of 	$\irr (N \mid M)$ by Theorem~\ref{Frobish}, the result follows.
\end{proof}

Note that if $G$ is an extraspecial $p$-group, $N$ is a maximal subgroup of $G$, and $M = Z(G)$ then $\psi^G \in \irr (G)$ for every character $\psi \in \irr (N \mid M)$. However $(G,N,M)$ is not a Frobenius triple since $G$ is a $p$-group. In particular, the converse of Theorem~\ref{Frobish} does not hold in general. However, as Theorem~\ref{introindcent} indicates, it does hold if we include an extra arithmetic condition.

\begin{proof}[Proof of Theorem~\ref{introindcent}]
Assume that $(\norm{G:N},\norm{M})=1$. If $(G,N,M)$ is a Frobenius triple, then $\psi^G\in\irr(G)$ for every $\psi\in\irr(N\mid M)$ by Theorem~\ref{Frobish}. Conversely, assume that $\psi^G\in\irr(G)$ for every $\psi\in\irr(N\mid M)$. Then $(G,N,M)$ is a Camina triple by Lemma~\ref{equivCam2}. Hence the result follows from Theorem~\ref{partconverse}.
\end{proof}

We now prove Theorem~\ref{FW thm}, which provides character-theoretic information about Frobenius--Wielandt triples $(G,H,L)$ in the special situation that $H$ has a normal complement, $M$. In the situation of Theorem~\ref{FW thm}, suppose that $L$ is chosen to be as small as possible. Then the Frobenius--Wielandt kernel $N$ satisfies conditions satisfied by Frobenius kernels. Of course, if $L=1$ then $N=M$ and $N$ is a Frobenius kernel. Outside of the situation $L=1$, since every character $\psi\in\irr(N\mid M)$ induces irreducibly to $G$ (alternatively since every $g\in G\setminus N$ acts fixed-point-freely on $M$), the size of the quotient group $N/M$ is in some sense a measure of how close $N$ is to being a Frobenius kernel.

\begin{proof}[Proof of Theorem~\ref{FW thm}]
Since $H\cap M=1$, we see that $H^g\cap M=1$ for all $g\in G$. Since $N=G\setminus\bigcup_{g\in G}(H\setminus L)^g$, we deduce that $M\le N$. So $N=HM\cap N=(H\cap N)M=LM$. Since $G=HN=HM$ and $H\cap M=1$, we see that $(G,(H\cap N)M,M)=(G,N,M)$ is a Frobenius triple by Theorem~\ref{FT FW equiv}. The remaining statement follows from Theorem~\ref{Frobish}.
\end{proof}

We conclude this section with an application. We use the above results to provide an analog for Frobenius--Wielandt triples of a well-known characterization of Frobenius complements. Fix a subgroup $H \le G$. Then $H$ is a Frobenius complement if and only if $N_G (X) \le H$ for every subgroup $1 < X \le H$. If $G$ is a group and $G=H\ltimes M$, then a similar necessary and sufficient condition exists for $(G,H,L)$ to be a Frobenius--Wielandt triple. 

\begin{lem}\label{normalizer}
Let $M$ be a normal subgroup of $G$ and suppose that $H \le G$ satisfies $G = HM$ and $H \cap M = 1$. Let $L \lhd H$ and assume that $N_G (X) \le H$ for every subgroup $X \le H$ satisfying $X \nleq L$. Then $(G,H,L)$ is a Frobenius--Wielandt triple, and hence $(G,LM,M)$ is a Frobenius triple.
\end{lem}

\begin{proof}
Let $g\in G$ and assume that $h\in (H\cap H^g)\setminus L$. Then there exists $k\in H$ and $x\in M$ such that $h=k^x$. Then $k^{-1}h=[k,x]\le H\cap M=1$ and so $h=k\in C_G(x)$. In particular, $x$ normalizes $\inner{h}$ Since $h\notin L$, it follows that $N_G(\inner{h})\le H$. Hence $x=1$ and $g\in H$. The last statement follows from Theorem~\ref{FW thm}.
\end{proof}

We now prove the theorem.

\begin{thm}\label{normalizer equiv}
Let $M$ and $N$ be proper, nontrivial normal subgroups of $G$ with $M \le N$. Let $H \le G$ and suppose that $G = HM$ and $H \cap M = 1$. Then $G$ satisfies the equivalent conditions of Theorem~$\ref{equiv}$ if and only if $N_G (X) \le H$ for every subgroup $X \le H$ satisfying $X \nleq N$.
\end{thm}

\begin{proof}
The one direction is Lemma~\ref{normalizer}.
To prove the converse, let $(G,N,M)$ be a Frobenius triple. Let $X\le H$ with $X\nleq H\cap N$ and let $g\in N_G(X)$. Then $X=X\cap X^g\le H\cap H^g$. Since $H^g\cap H\le H\cap N$ if $g\notin H$, by Theorem~\ref{equiv}, and since $X$ cannot be contained in $H\cap N$, we conclude that $g\in H$.
\end{proof}

The example $G = C_3 \times S_3$ from earlier shows that is necessary to assume that $G=HM$ in Theorem~\ref{normalizer equiv}.

\section{Multiple Frobenius triples}

We now consider ways of constructing new Frobenius triples from existing ones. 

\begin{lem}\label{frobquotient 2}
Let $(G,N,M)$ be a Frobenius triple, and let $1 < K$ be a normal subgroup of $G$ contained in $N$ satisfying $M\cap K<M$. 	Then $(G/K,N/K,MK/K)$ is a Frobenius triple. In particular, if $N=M\times K$ then $G/K$ is a Frobenius group with 	Frobenius kernel $N/K$.
\end{lem}

\begin{proof}
Let $\psi\in\irr(G/K\mid MK/K)$. Then,considering $\psi$ as a character of $G$, observe that $M\nleq\ker(\psi)$ and so $\psi^G\in\irr(G)$. Thus $\psi^{G/K}\in\irr(G/K)$. Also $\norm{MK/K}=\norm{M/(M\cap K)}$ divides $\norm{M}$, which is relatively prime to $\norm{G:N}=\norm{G/K:N/K}$. Thus $(G/K,N/K,MK/K)$ is a Frobenius triple by Theorem~\ref{introindcent}.
\end{proof}

We have the following immediate consequence when $K < M$.

\begin{cor}\label{centralizer}
Let $(G,N,M)$ be a Frobenius triple. If $1 < K < M$ is normal in $G$, then $(G/K,N/K,M/K)$ is a Frobenius triple.
\end{cor}

In the next several results, we consider what happens when a group contains two different Frobenius triples. We first consider products of the bottom factors.

\begin{lem}\label{two Ms}
Let $M_1,M_2$ and $N$ be normal subgroups of $G$ satisfying $M_1M_2\le N$. If $(G,N,M_1)$ and $(G,N,M_2)$ are Frobenius triples, then so is $(G,N,M_1M_2)$. 
\end{lem}

\begin{proof}
First observe that $\irr(N\mid M_1M_2)=\irr(N\mid M_1)\cup\irr(N\mid M_2)$. Thus every $\psi\in\irr(N\mid M_1M_2)$ induces irreducibly to $G$. Since $(\norm{G:N},\norm{M_i})=1$ for $1,2$, we see that $(\norm{G:N},\norm{M_1M_2})=1$. Hence $(G,N,M_1M_2)$ is a Frobenius triple. 
\end{proof}

We obtain the following interesting consequence of Lemma~\ref{two Ms}.

\begin{cor}\label{centralizer product}
Let $L_1$ and $L_2$ be normal subgroups of $G$. Then $\inner{C_G(x):1\ne x\in L_1L_2}=\inner{C_G(x):1\ne x\in L_1\cup L_2}$.
\end{cor}

\begin{proof}
Write $C_i=\inner{C_G(x):1\ne x\in L_i}$ for $1=1,2$ and $C=\inner{C_G(x):1\ne x\in L_1\cup L_2}=C_1C_2$. Write $N=\inner{C_G(x):1\ne x\in L_1L_2}$. It is easy to see that $C\le N$. If $C=G$, then we are done so assume that $C<G$. Then $(G,C,L_1)$ and $(G,C,L_2)$ are Frobenius triples. By Lemma~\ref{two Ms} $(G,C,L_1L_2)$ is also a Frobenius triple; thus $C_G(x)\le C$ for every $1\ne x\in L_1L_2$. Hence $N\le C$, as required.
\end{proof}

We next consider the case when two different quotients are Frobenius triples.  The most important case is when $K_1 \cap K_2 = 1$.

\begin{lem}\label{intersection}
Let $M$ and $N$ be normal subgroups of $G$ satisfying $M\le N$. Let $K_1$ and $K_2$ be normal subgroups of $G$ properly contained in $M$. If $(G/K_i,N/K_i,M/K_i)$ is a Frobenius triple for $i=1,2$, then so is $(G/(K_1\cap K_2),N/(K_1\cap K_2),M/(K_1\cap K_2))$.
\end{lem}

\begin{proof}
Write $K_{12}=K_1\cap K_2$. Let $g\in G\setminus N$. Let $D_i=\{x\in G:[g,x]\in K_i\}$ for $i=1,2$, let $D_{12}=\{x\in G:[g,x]\in K_{12}\}$, and observe that $D_{12}=D_1\cap D_2$. Since $(G/K_i,N/K_i,M/K_i)$ is a Frobenius triple for $i=1,2$, it follows that $D_i\cap M\le K_i$ for $i=1,2$. Hence, $D_{12}\cap M=(D_1\cap D_2)\cap M\le K_1\cap K_2=K_{12}$, and we deduce that $C_{M/K_{12}}(gK_{12})=D_{12}/K_{12}\cap M/K_{12}$ is trivial. It follows that $(G/K_{12},N/K_{12},M/K_{12})$ is a Frobenius triple, as desired.
\end{proof}

We now return to the case where we have two Frobenius triples, and we consider the intersection of the bottom subgroups.

\begin{lem}\label{two trips}
If $(G,N_1,M_1)$ and $(G,N_2,M_2)$ are Frobenius triples and $M_1 \cap M_2 > 1$, then $(G,N_1\cap N_2,M_1\cap M_2)$ is a Frobenius triple. 
\end{lem}

\begin{proof}
For each element $1\ne x\in M_1\cap M_2$, we have $C_G(x) \le N_1$ since $x \in M_1$ and $(G,N_1,M_1)$ is a Frobenius triple; similarly $C_G(x)\le N_2$. 
\end{proof}

\section{A Galois connection}\label{galois}

Let $G$ be a group. For a character $\chi$ of $G$, the {\it vanishing-off} subgroup of $\chi$---denoted $V(\chi)$---is the smallest  subgroup $V$ of $G$ such that $\chi(g)=0$ for all $g\in G\setminus V$. For a nontrivial, proper normal subgroup $M$ of $G$ Mlaiki \cite{NM14} defines the subgroup $V(G\mid M)$ to be the product of the subgroups $V(\chi)$ for $\chi\in\irr(G\mid M)$. This subgroup gives a convenient way of characterizing Camina triples. Mlaiki shows that $(G,N,M)$ is a Camina triple if and only if $V(G\mid M)\le N$.

In \cite{SBMLnested} the authors define another subgroup $U(G\mid N)$ for any normal subgroup $N$ of $G$ that can also be used to characterize Camina triples. For $M,N\lhd G$, we show that $V(G\mid M)\le N$ if and only if $M\le U(G\mid N)$. This not only gives the alternate characterization, but also shows that the two maps $V(G\mid\rule{.25cm}{.5pt})$ and $U(G\mid\rule{.25cm}{.5pt})$ provide a (monotone) Galois connection on the lattice of normal subgroups of the group $G$. In the section we explore a similar Galois connection coming from Frobenius triples.

We define $C (G \mid M) = \langle C_G (x) \mid 1 \ne x \in M \rangle$.  It is not difficult to see that if $N$ is a proper normal subgroup of $G$ containing $M$, then $(G,N,M)$ is a Frobenius triple if and only if $C (G \mid M) \le N$.  Furthermore, note that $M$ will be a Frobenius kernel if and only if $C (G \mid M) = M$. 

For $N\lhd G$, define $I(G\mid N)=\prod_{M\in\mathcal{F}}M$, where $\mathcal{F}=\{M\lhd G:C(G\mid M)\le N\}$. By Lemma~\ref{two Ms}, it is clear that $(G,N,M)$ is a Frobenius triple if and only if  $M\le I(G\mid N)$. It is also clear that $N$ will be a Frobenius kernel if and only if $N=I(G\mid N)$.  

We summarize these results below, as well as some other basic properties that hold in general for Galois connections.

\begin{lem}\label{galois basics}
Let $1 < M, M_1, M_2, N, N_1, N_2 < G$ be normal subgroups of $G$. The following hold:
\begin{enumerate}[label={\rm(\arabic*)}]\openup3pt
		\item $C(G\mid M)\le N$ if and only if $M\le I(G\mid N)$;
		\item $C(G\mid M_1M_2)=C(G\mid M_1)C(G\mid M_2)$;
		\item $I(G\mid N_1\cap N_2)=I(G\mid N_1)\cap I(G\mid N_2)$;
		\item $M\le I(G\mid C(G\mid M))$ and $C(G\mid I(G\mid N))\le N$;
		\item $Z(G)\cap I(G\mid N)=1$;
		\item $(G,N,M)$ is a Frobenius triple if and only if $C(G\mid M)\le N$;
		\item $G$ is a Frobenius group if and only if $C(G\mid F(G))=F(G)$, which happens if and only if $F(G)=I(G\mid F(G))$;
		\item $Z(G)M\le V(G\mid M)\le C(G\mid M)$ and $I(G\mid N)\le U(G\mid N)\le N\cap G'$.
\end{enumerate}
\end{lem}

\begin{proof}
Statement (1) is clear from the definitions. Since $C(G\mid M_1)C(G\mid M_2)=\inner{C_G(x):1\ne x\in M_1\cup M_2}$, statement (2) is just Corollary~\ref{centralizer product}. One direction of statement (3) is clear from the definition; the other follows immediately from Lemma~\ref{two trips}. Statement (4) follows immediately from statement (1). Since every $g\in G\setminus N$ acts fixed-point-freely on $I(G\mid N)$, statements (5), (6) and (7) follow.

We now prove (8). Since $(G,C(G\mid M),M)$ is a Frobenius triple, it is also a Camina triple by Lemma~\ref{camfrob}. Thus $V(G\mid M)\le C(G\mid M)$ follows from \cite[Theorem 2.1]{NM14}. Similarly $I(G\mid N)\le U(G\mid N)$ \cite[Lemma 5.8]{SBMLnested}. The remaining two containments follow from Lemmas 5.2, 5.3 and 5.4 of \cite{SBMLnested} and Lemma 2.4 of \cite{NM14}.
\end{proof}

The next result gives two alternate descriptions of $I(G\mid N$). As is standard, when $\psi$ is an irreducible character of a normal subgroup $N$ of $G$, we let $I_G(\psi)$ denote the stabilizer of $\psi$ under the natural action of $G$ on $\irr(N)$ induced by conjugation. 

\begin{lem}\label{I(G|N) description}
Let $1<M,N<G$ be normal subgroups of $G$. Write $\pi=\pi(G/N)$. The following hold:
\begin{enumerate}[label={\rm(\arabic*)}]\openup3pt
		\item A normal subgroup $M$ of $G$ satisfies $M\le I(G\mid N)$ if and only if $M\le O_{\pi'}(G)$ and 
		$\psi^G\in\irr(G)$ for every $\psi\in\irr(N\mid M)$; 
		\item $I(G\mid N)=O_{\pi'}(G)\cap\bigcap\limits_{\substack{\psi\in\irr(N)\\ I_G(\psi)>N}}\ker(\psi)$.
\end{enumerate}
\end{lem}

\begin{proof}
Statement (1) is a restatement of Theorem~\ref{introindcent}. 

Since $(G,N,I(G\mid N))$ is a Frobenius triple, $I(G\mid N)\le O_{\pi'}(G)$ by Lemma~\ref{partconverse 1}. Let $K$ be the intersection of the kernels of all irreducible characters $\psi$ of $N$ satisfying $I_G(\psi)>N$. Observe that every $\chi\in\irr(N\mid I(G\mid N))$ induces irreducibly to $G$ by Theorem~\ref{introindcent}. It is easy to see that this is equivalent to $I_G(\psi)=N$ for all $\psi\in\irr(N\mid I(G\mid N))$. Since $I_G(\psi)=N$ if and only if $\psi\in\irr(G\mid K)$, we see that $\irr(G\mid I(G\mid N))\subseteq\irr(G\mid K)$, which gives $I(G\mid N)\le K$. Now let $L=O_{\pi'(G)}\cap K$. Then $\irr(G\mid L)\subseteq\irr(G\mid K)$, so every $\psi\in\irr(G\mid L)$ induces irreducibly to $G$. Since $L\le O_{\pi'}(G)$, it follows from Theorem~\ref{introindcent} that $(G,N,L)$ is a Frobenius triple. Thus $L\le I(G\mid N)$ by Lemma~\ref{galois basics}.
\end{proof}

\section{Examples}

We now present a number of examples to see what can occur in Frobenius triples. The first result tells us that we cannot expect to say anything about the structure of $N/M$ when $(G,N,M)$ is a Frobenius triple. 

\begin{lem} \label{lem}
Let $(G,N,M)$ be a Frobenius triple and let $A$ be any group. Then $(G\times A,N\times A,M)$ is a Frobenius triple.
\end{lem}

\begin{proof}
Consider any element $1\ne x\in M$. Then $C_{G\times A}(x) = C_G(x)\times A\le N\times A$.
\end{proof}

If $A$ is any group, then there is a Frobenius triple $(G,N,M)$ so that $N/M \cong A$ by taking $H$ to be a Frobenius group with Frobenius kernel $M$.  By Lemma \ref{lem}, we see that $(H \times A, M \times A, M)$ is a Frobenius triple; so take $G = H \times A$ and $N = M \times A$.

We next show that if we have a group of prime order acting coprimely on an abelian group, then the resulting semidirect product yields a Frobenius triple.

\begin{exam}
Let $N$ be an abelian group and $H$ be a cyclic group of prime order not dividing $\norm{N}$. Suppose that $H$ acts on $N$ and let $G=H\ltimes N$. We claim that $(G,N,G')$ is a Frobenius triple. To see this, first note that $(\norm{G:N},\norm{G'})=1$. Let $\psi\in\irr(N\mid G')$. Since $G'\nleq\ker(\psi)$, $\psi$ does not extend to $G$. Thus $\psi^G$ is irreducible and $(G,N,G')$ is a Frobenius triple, as claimed.
\end{exam}

We next present an example of a Frobenius triple that is not split over the bottom subgroup.

\begin{exam}
Let $\mathbb{F}$ be the finite field of cardinality 64, and let $K\le\mathrm{GL}_3(\mathbb{F})$ be the subgroup of all upper triangular unipotent matrices. Let $D\le \mathrm{GL}_3(\mathbb{F})$ be the subgroup of all diagonal matrices of the form $\mathrm{diag}(1,a,a^2)$, where $a\in \mathbb{F}^\times$. Let $H$ be a Sylow $3$-subgroup of $D$. Then $G=HK$ is a Frobenius group with Frobenius kernel $K$. Let $M=K'$, and let $N\le G$ satisfy $\norm{N:K}=3$. Then $(G,N,M)$ is a Frobenius triple that does not split over $M$, since $M\le \Phi(G)$.
\end{exam}

Next, we consider a class of examples that has arisen in a number of places in the literature.

\begin{exam}
Let $(G,N)$ be a Camina pair with $G/N$ a $p$-group for some prime $p$. Assume that $G$ is not a $p$-group and that $(G,N)$ is not a Frobenius group. Chillag and Macdonald conjectured in \cite{genFrobGps1} that $G$ is a Frobenius group with Frobenius complement isomorphic to the quaternion group $Q_8$ and $\norm{G:N}=4$. This conjecture is verified by Dark and Scoppola in \cite{CaminaGpClass} in the case that $G/N$ is abelian. Isaacs \cite{MICoprime} showed that $G$ must have a solvable normal $p$-complement $M$ and that $C_G(x)\le N$ for every nonidentity element $x\in M$. Around the same time, Chillag, Mann and Scoppola independently discovered similar results in \cite{genFrobGps2} and showed that $(G,P,P\cap N)$ is a Frobenius--Wielandt group with Frobenius--Wielandt kernel $N$, where $P$ is a Sylow $p$-subgroup of $G$. In particular we see that $(G,N,M)$ is a Frobenius triple.
 \end{exam}
 
%
%
%

We know that when $G$ is a Frobenius group, the set of Frobenius complements are conjugate.  It is natural to ask what can be said about the pseudo-complements for a Frobenius triple.  We present two examples to show that very little can be said.  The first shows that there is not a unique conjugacy class of such pseudo-complements. In fact, it also established that there is not even a unqiue conjugacy class of maximal pseudo-complements.  

\begin{exam}
Take $G=A\times B$, where $A=B=S_3$, $M=A'$, and $N=A'\times B$.  It is not difficult to see that if $x \in M \setminus \{ 1 \}$, then $C_G (x) = N$.  It follows that $(G,N,M)$ is a Frobenius triple.  Take $D$ to be the diagonal subgroup of $G$; observe that the only maximal subgroup containing $D$ is $DM$. (This can be verified this with MAGMA). So $D$ is itself maximal subject to the condition that $G = DN$ and $D\cap M=1$. Let $H=P\times B$, where $P$ is a Sylow $2$-subgroup of $G$. It is not difficult to see that $H$ is a pseudo-complement of $G$.  Then $\lvert{H}\rvert=12$ is the maximum order of a pseudo-complement. It turns out that every pseudo-complement of order $12$ is conjugate to $H$. (This is also checked with MAGMA). So there is not a unique conjugacy class of maximal pseudo-complements, and there can exist conjugacy classes of pseudo-complements of different orders. 
\end{exam}

The next example shows that there can even be two different conjugacy classes of pseudo-complements of maximal possible order.

\begin{exam}
Take $G$ to be the Frobenius group $C_2\ltimes(C_3\times C_3)$. Let $P$ be a Sylow $2$-subgroup and let $M$ be minimal normal. Let $L_1, L_2, L_3$ be the remaining minimal normal subgroups of $G$. Let $H_i = PL_i$ for $i = 1, 2, 3$. Let $N = F(G)$.  Then $\lvert{M}\rvert = 3$ and $\lvert{G:N}\rvert = 2$. Also, $G=H_iN$ and $H_i\cap M = 1$ for $i = 1, 2, 3$.  Note that $H_i$ and $H_j$ are not conjugate if $i \ne j$. Thus there is not a unique conjugacy class of pseudo-complements of maximal order. Also, note that $(G,N,M)$ is a Frobenius triple, so this does not even hold for them. 
\end{exam}


\end{document}